\documentclass[12pt, reqno]{amsart}
\usepackage{amsmath, amsthm, amscd, amsfonts, amssymb, graphicx, color}
\usepackage{mathrsfs}
\usepackage{graphicx}
\usepackage{caption}
\usepackage{subcaption}
\usepackage{enumerate}
\usepackage[bookmarksnumbered, colorlinks, plainpages]{hyperref}
\hypersetup{colorlinks=true,linkcolor=red, anchorcolor=green, citecolor=cyan, urlcolor=red, filecolor=magenta, pdftoolbar=true}

\newtheorem{theorem}{Theorem}[section]

\newtheorem{question}{Question}[section]

\newtheorem{example}{Example}[section]
\newtheorem{lemma}{Lemma}[section]
\newtheorem{corollary}{Corollary}[section]
\newcommand\restr[2]{{% we make the whole thing an ordinary symbol
  \left.\kern-\nulldelimiterspace % automatically resize the bar with \right
  #1 % the function
  \littletaller % pretend it's a little taller at normal size
  \right|_{#2} % this is the delimiter
  }}

\newcommand{\littletaller}{\mathchoice{\vphantom{\big|}}{}{}{}}

\date{}
\begin{document}

\title [Iterated Aluthge transforms]{Iterated Aluthge transforms of some composition operators on weighted Bergman spaces}

\author[S. Lahiri, S. Ojha, R. Birbonshi]{Sudeshna Lahiri, Sarita Ojha, Riddhick Birbonshi}
	
\address[Lahiri]{Department of Mathematics, Indian Institute of Engineering Science and Technology, Shibpur, Howrah 711103, West Bengal, India}
\email{sudeshna.lahiri88@gmail.com}
	
\address[Ojha]{Department of Mathematics, Indian Institute of Engineering Science and Technology, Shibpur, Howrah 711103, West Bengal, India}
\email{sarita.ojha89@gmail.com}

\address[Birbonshi] {Department of Mathematics, Jadavpur University, Kolkata 700032, West Bengal, India}
\email{riddhick.math@gmail.com}



% \begin{center}
%     \Large{Iterated Aluthge transforms of some composition operators on weighted Bergman spaces}\\\vspace{0.5cm}\small{Sudeshna Lahiri$^1$, Sarita Ojha$^*$\footnote{Corresponding author. Email: sarita.ojha89@gmail.com\\
%     Email address: sudeshna.lahiri88@gmail.com (S. Lahiri), sarita.ojha89@gmail.com (S. Ojha), riddhick.math@gmail.com (R. Birbonshi)}, Riddhick Birbonshi$^2$}\\
%     $^1$ Department of Mathematics, Indian Institute of Engineering Science and Technology, Shibpur, Howrah 711103, India\\
%     $^2$ Department of Mathematics, Jadavpur University, Kolkata 700032, West Bengal, India.
%     \end{center}

\begin{abstract}
In this paper, we compute the iterated Aluthge transforms $\widetilde{C_\phi}^{(n)}$ of the composition operator $C_\phi$ on the weighted Bergman spaces $\mathcal{A}_\alpha^2(\mathbb{D})$, where $\phi(z)=az+(1-a)$ for $0<a<1$. Also, we obtain the norm and numerical radius of $\widetilde{C_\phi}^{(n)}$ on $\mathcal{A}_\alpha^2(\mathbb{D})$. We establish that $\widetilde{C_\phi}^{(n)}$ converges in the strong operator topology on $\mathcal{A}_\alpha^2(\mathbb{D})$. The purpose of this paper is to examine the results of \cite{jung2015iterated} for the weighted Bergman spaces $\mathcal{A}_\alpha^2(\mathbb{D})$. Additionally, by using the iterated Aluthge transforms of $C_\phi^*$ on $\mathcal{A}_\alpha^2(\mathbb{D})$, we derive the iterated Aluthge transforms of $C_\sigma$, where $\displaystyle\sigma(z)=\frac{az}{-(1-a)z+1}$ for $0<a<1$, on some weighted Hardy space $H^2(\beta_\alpha)$ and study its convergence. Finally, we raise some questions that emerge from these findings.
\end{abstract}

\maketitle

% Keywords: Composition operator, Weighted composition operator, Aluthge transform, Weighted Bergman space.\\
% Mathematics Subject Classification 2020: 47B32, 47B33

\section{Introduction}
 Let $H$ be a separable complex Hilbert space and $B(H)$ denote the algebra of all bounded linear operators on $H$. Then every operator $T \in B(H)$ has a unique polar decomposition $T=U|T|$ where $|T| = \sqrt{T^* T}$ and $U$ is the partial isometry satisfying ker$(T)$=ker$(U)$=ker$(|T|)$. Associated with $T$, there is a very useful operator $$\widetilde{T}=|T|^{1/2} U |T|^{1/2},$$ called the Aluthge transform of $T$ introduced by Aluthge \cite{aluthge1990p}. Also, we denote the iterated Aluthge transform of $T$ by 
 \begin{equation*}
     \widetilde{T}^{(0)} =T \mbox{ and } \widetilde{T}^{(n+1)} = \widetilde{(\widetilde{T}^{(n)})}
 \end{equation*} 
 for every non-negative integer $n$. Many researchers have studied this transform in the literature (see \cite{antezana2011iterated,bong2000aluthge, bong2003iterated, cowen2010hermitian, ito2004polar, jung2015iterated,muneo2002spectral}, etc.).\\

    % In this paper, we study the iterated Aluthge transform of composition operator $C_\phi$ and $C_\sigma$ on weighted Bergman space where $\phi(z)=az+(1-a)$ and $\sigma(z)=\frac{az}{-(1-a)z+1}$ for $0<a<1$.

Let $\mathbb{D}$ denote the open unit disk in the complex plane. The weighted Bergman space $\mathcal{A}_\alpha ^2(\mathbb{D)}$, for $\alpha>-1$, is defined by 
\begin{align*}
    \mathcal{A}_\alpha ^2(\mathbb{D)}=\left\{f \ \text{analytic in}\  \mathbb{D}:\|f\|_\alpha^2 = \int_\mathbb{D}|f(z)|^2 \frac{dA_\alpha(z)}{\pi} < \infty\right\},
\end{align*}
with $dA_\alpha(z)=(\alpha+1)(1-|z|^2)^\alpha dA(z)$ where $dA$ is the normalized area measure. When $\alpha=0$, the weighted Bergman space reduces to the classical Bergman space $\mathcal{A}^2(\mathbb{D)}$. For $f,g\in \mathcal{A}_\alpha ^2(\mathbb{D)}$, the inner product on $\mathcal{A}_\alpha^2(\mathbb{D})$ is given by
\begin{align*}
    \langle f,g\rangle= \int_\mathbb{D}f(z)\overline{g(z)}(\alpha+1)(1-|z|^2)^\alpha\frac{dA(z)}{\pi}.
\end{align*}
For any non-negative integer $n$, $\mathcal{A}_\alpha^2(\mathbb{D})$ has the orthonormal basis $\{e_n\}_{n=0}^\infty$ where
$$e_n(z)=\sqrt{\frac{\Gamma(n+\alpha+2)}{n!\Gamma(\alpha+2)}}z^n, \hspace{5mm} z\in \mathbb{D}.$$
Also for $\displaystyle f(z)=\sum_{n=0}^\infty a_n z^n$, $\displaystyle g(z)=\sum_{n=0}^\infty b_n z^n\in \mathcal{A}_\alpha ^2(\mathbb{D)}$, the inner product on $\mathcal{A}_\alpha^2(\mathbb{D})$ can also be expressed as 
$$\langle f,g\rangle=\sum_{n=0}^\infty\frac{n!\Gamma(\alpha+2)}{\Gamma(n+\alpha+2)}a_n \overline{b_n}.$$
It is well known that $\mathcal{A}_\alpha ^2(\mathbb{D)}$ is a reproducing kernel Hilbert space with reproducing kernel function given by
\begin{equation}\label{K_w}
    K_\omega(z)=\frac{1}{(1-\overline{\omega}z)^{\alpha+2}}
\end{equation} with 
\begin{equation}\label{norm K_w}
    \|K_\omega\|^2=\frac{1}{(1-|\omega|^2)^{\alpha+2}}.
\end{equation}
For more details about the weighted Bergman space, one may refer to \cite{hedenmalm2012theory}.
The Hardy-Hilbert space, denoted by $H^2$, consists of all analytic functions on $\mathbb{D}$ having power series representations with square-summable complex coefficients. The space $L^\infty(dA_\alpha)$ consists of all essentially bounded measurable functions $\phi$ on $\mathbb{D}$ such that
\begin{align*}
    \|\phi\|_{\infty,\alpha}=\sup\{c\geq 0 : A_\alpha(\{z\in \mathbb{D}:|\phi(z)|>c\})>0\}<\infty.
\end{align*} The space $H^\infty$ denotes the set of bounded and analytic functions on $\mathbb{D}$. For $f \in L^\infty(dA_\alpha)$, the Toeplitz operator $T_f$ is defined by $T_f h = P_\alpha(fh)$ for all $h \in \mathcal{A_\alpha}^2(\mathbb{D)}$, where $P_\alpha$ denotes the orthogonal projection of $L^2$ onto $\mathcal{A}_\alpha^2(\mathbb{D)}$. It is easy to see that $T_f$ is a bounded linear operator on $\mathcal{A}_\alpha ^2(\mathbb{D})$ with $\|T_f\|\leq \|f\|_{\infty,\alpha}$. If $f$ is analytic in $\mathbb{D}$, and $\phi$ is an analytic map of $\mathbb{D}$ into itself, the weighted composition operator on $\mathcal{A}_\alpha ^2(\mathbb{D)}$ with symbols $f$ and $\phi$ is defined by $$W_{f,\phi} h=f.(h \circ \phi)$$ for all $h$ in $\mathcal{A}_\alpha ^2(\mathbb{D)}$. For $f$ in $H^\infty$, the weighted composition operator $W_{f,\phi}$ is bounded on $\mathcal{A_\alpha}^2(\mathbb{D)}$. In particular, when $f=1$, $W_{f,\phi}$ becomes the composition operator $C_\phi$ and when $\phi$ is the identity map on $\mathbb{D}$, $W_{f,\phi}$ becomes the multiplication operator. For analytic self map $\phi$ on $\mathbb{D}$, $C_\phi$ is bounded on $\mathcal{A}_\alpha^2(\mathbb{D})$ by Lemma 2.3 of \cite{richman2003subnormality}. For any $f \in H^\infty$ and analytic self map $\phi$ in $\mathbb{D}$, the adjoint of $W_{f,\phi}=T_f C_\phi$ satisfies
\begin{equation*}
    W_{f,\phi}^*K_\omega=C_\phi ^* T_f ^* K_\omega=\overline{f(\omega)}K_{\phi(\omega)}
\end{equation*} for any $\omega \in \mathbb{D}$. In particular, $C_\phi^*K_\omega=K_{\phi(\omega)}$ and $T_f^*K_\omega=\overline{f(\omega)}K_\omega$.
Throughout this paper, we use the following properties of Toeplitz and composition operators:
If $\rho$ and $\tau$ are analytic self maps on $\mathbb{D}$, $\eta$ and $\zeta$ are functions in $H^\infty$, then 
\begin{equation*}
    C_\rho C_\tau=C_{\tau\circ \rho},\ T_\eta T_\zeta=T_{\eta.\zeta},\ C_\rho T_\eta=T_{\eta \circ \rho}C_\rho\ .
\end{equation*}
For more details on composition operator, we refer to the book \cite{cowen1995composition}.

%In \cite{cowen1988linear}, Cowen gave an expression for the adjoint formula of a composition operator on $H^2$ whose symbol is a linear fractional self map of $\mathbb{D}$. 
%\begin{theorem}\cite{cowen1988linear}\label{Theorem1.1}
  %  Let $\phi= \frac{az+b}{cz+d}$ be a linear fractional transformation mapping $\mathbb{D}$ into itself, where $ad-bc \neq 0$, then $\sigma(z)=\frac{\overline{a}z-\overline{c}}{-\overline{b}z +\overline{d}}$ maps $\mathbb{D}$ into itself, $\delta(z)=\frac{1}{-\overline{b}z +\overline{d}}$ and $h(z)=cz+d$ are in $H^\infty$, and $$C_\phi^* =T_\delta C_\sigma T_h^*.$$ 
%\end{theorem}

Throughout this article, we use two analytic self maps $\phi(z)$ and $\sigma(z)$ on $\mathbb{D}$ where $\phi(z)=az+(1-a)$ and $\sigma(z)=\frac{az}{-(1-a)z+1}$ for $0<a<1$.
In \cite{jung2015iterated}, Jung et al. investigated the iterated Aluthge transforms of the composition operators $C_\phi$ and $C_\sigma$ on $H^2$. Motivated by this, we examine their results in the weighted Bergman space $\mathcal{A}_\alpha^2(\mathbb{D})$ for $\alpha>-1$. In Section \ref{section2}, we express the iterated Aluthge transforms of $C_\phi$ on $\mathcal{A}_\alpha^2(\mathbb{D)}$ as some weighted composition operators. We also present results concerning the norm and numerical range of $\widetilde{C}_\phi^{(n)}$ for $n\in \mathbb{N}\cup\{0\}$. In Section \ref{section3}, we prove that the iterated Aluthge transform of $C_\phi$ on $\mathcal{A}_\alpha^2(\mathbb{D)}$ converges in the strong operator topology to a normal operator, but fails to converge in norm topology. In Section \ref{section4}, we compute the iterated Aluthge transform of $C_\sigma$ on some weighted Hardy space $H^2(\beta_\alpha)$ where
\begin{eqnarray*}
    \beta_\alpha(j) &=&\begin{cases}
            1,& j=0\\
            \sqrt{\frac{\Gamma(j+\alpha+1)}{{(j-1)!\Gamma(\alpha+2)}}},&j\geq 1.
        \end{cases}
\end{eqnarray*}
by using the iterated Aluthge transform of $C_\phi^*$ on $\mathcal{A}_\alpha^2(\mathbb{D)}$. Finally, in Section \ref{section5}, we outline some questions that arise naturally from our results.

\section{Iterated Aluthge transforms of $C_\phi$ on $\mathcal{A}_\alpha ^2(\mathbb{D)}$}\label{section2}
In this section, we derive the iterated Aluthge transform of $C_\phi$ on $\mathcal{A}_\alpha^2(\mathbb{D)}$ where $\phi(z)=az+(1-a)$ for $0<a<1$. With the help of this, we examine some classes of weighted composition operators on $\mathcal{A}_\alpha^2(\mathbb{D)}$.
In \cite{hurst1997relating}, Hurst has given an expression for the adjoint of a composition operator on $\mathcal{A}_\alpha^2(\mathbb{D})$ where the symbol is a linear fractional self map of $\mathbb{D}$ as follows:
\begin{theorem}\cite{hurst1997relating}\label{Theorem2.1}
    Let $\phi= \frac{az+b}{cz+d}$ be a linear fractional transformation mapping $\mathbb{D}$ into itself, where $ad-bc \neq 0$. Let $\alpha>-1.$ Then $\sigma(z)=\frac{\overline{a}z-\overline{c}}{-\overline{b}z +\overline{d}}$ maps $\mathbb{D}$ into itself, $\delta(z)=\frac{1}{(-\overline{b}z +\overline{d})^{\alpha+2}}$ and $h(z)=(cz+d)^{\alpha+2}$ are in $H^\infty$, and as operators on $\mathcal{A}_\alpha ^2(\mathbb{D)}$, $$C_\phi^* =T_\delta C_\sigma T_h^*.$$
\end{theorem}
\noindent Using the previous theorem, we give the following lemma.
\begin{lemma}\label{lemma2.1}
    Let $\phi= \frac{az+b}{cz+d}$ be a linear fractional transformation mapping $\mathbb{D}$ into itself with $ad-bc \neq 0$. Then 
    $$C_\phi ^* T_r^*=T_\delta C_\sigma,$$ where $\sigma(z)=\frac{\overline{a}z-\overline{c}}{-\overline{b}z +\overline{d}}$, $\delta(z)=\frac{1}{(-\overline{b}z +\overline{d})^{\alpha+2}}$ and $r(z)=\frac{1}{(cz+d)^{\alpha+2}}$ as operators on $\mathcal{A}_\alpha ^2(\mathbb{D)}$. 
\end{lemma}
\begin{proof}
    From Theorem \ref{Theorem2.1}, $C_\phi^* =T_\delta C_\sigma T_h^*,$ where $\sigma(z)=\frac{\overline{a}z-\overline{c}}{-\overline{b}z +\overline{d}}$, $\delta(z)=\frac{1}{(-\overline{b}z +\overline{d})^{\alpha+2}}$ and $h(z)=(cz+d)^{\alpha+2}$. Thus $C_\phi^*T_r^*=T_\delta C_\sigma T_h^*T_r^*=T_\delta C_\sigma
    $.
\end{proof}

Let $\textit{P}=\{t:Re (t)>0\}$ denote the right half plane. For $t$ in $\textit{P}$ and $\alpha>-1$, let
\begin{equation}\label{eq3}
 A_t=T_{f_t}C_{\phi_t} \ \mbox{where } f_t(z)=\frac{1}{(1+t-tz)^{\alpha+2}} \ \mbox{and } \ 
    \phi_t(z)=\frac{t+(1-t)z}{1+t-tz}.
\end{equation}
    
\noindent In Theorem 13 of \cite{cowen2013hermitian}, it is shown that $A_t$, for $t$ in $\textit{P}$, forms an analytic semigroup of weighted composition operators on $\mathcal{A}_\alpha^2(\mathbb{D)}$. The boundedness of $A_t$ for $t>0$ follows from Corollary 18 of \cite{cowen2013hermitian}.\\

\begin{lemma}\label{Lemma2.2}
     For any positive numbers $p$ and $s$, let $\sigma_s(z)=e^{-s}z+1-e^{-s}$ and $t=e^s-1$. Then 
     $$(C_{\sigma_s}^* C_{\sigma_s})^p=e^{ps(\alpha+2)}A_{pt} \ ,$$ 
     where $A_t=T_{f_t} C_{\phi_t}$ given by \eqref{eq3} and the composition operator $C_{\sigma_s}$ are defined on $\mathcal{A}_\alpha ^2(\mathbb{D)}$.
\end{lemma}
\begin{proof}
    For $\sigma_s(z)=e^{-s}z+1-e^{-s}$, by Theorem \ref{Theorem2.1}, we have
    $$C_{\sigma_s}^*=T_gC_\psi T_h^*$$ where $g(z)=\frac{1}{(-(1-e^{-s})z+1)^{\alpha+2}}$, $\psi(z)=\frac{e^{-s}z}{-(1-e^{-s})z+1}$ and $h(z)=(0z+1)^{\alpha+2}=1$. So $T_h^*=I$ and hence
    $$C_{\sigma_s}^*C_{\sigma_s}=T_gC_\psi C_{\sigma_s}=T_g C_{\sigma_s \circ \psi}$$
    where 
    \begin{eqnarray*}
        \sigma_s \circ \psi &=& \frac{t+(1-t)z}{1+t-tz}=\phi_t(z)\\
        g(z) &=&\frac{(e^s)^{\alpha +2}}{(-(e^s-1)z+e^s)^{\alpha+2}}
        =\frac{e^{s(\alpha+2)}}{(-tz+t+1)^{\alpha+2}}=e^{s(\alpha+2)}f_t(z).
    \end{eqnarray*}
 Thus we have
    $$C_{\sigma_s}^*C_{\sigma_s}=T_g C_{\sigma_s \circ \psi}=e^{s(\alpha+2)}T_{f_t}C_{\phi_t}=e^{s(\alpha+2)}A_t.$$
    By semigroup property of $A_t$, we have $(A_t)^p=A_{pt}$, hence $(C_{\sigma_s}^* C_{\sigma_s})^p=e^{ps(\alpha+2)}A_{pt}$.
\end{proof}

\begin{lemma}\label{lemma2.3}
    For $\alpha>-1$ and $s>0$, the operator $$U=C_{\frac{(e^s+1)z+e^s-1}{(e^s-1)z+e^s+1}}T_{\left(\frac{(1-e^s)z+e^s+1}{2e^{s/2}}\right)^{\alpha+2}}\in B(\mathcal{A}_\alpha^2(\mathbb{D))}$$ is unitary.
\end{lemma}
\begin{proof}
    Since $T_{\left(\frac{(1-e^s)z+e^s+1}{2e^{s/2}}\right)^{\alpha+2}}$ is a bounded analytic Toeplitz operator, the operator $U$ is bounded by Theorem 6.4 of \cite{cowen2010hermitian}. Let $$c=\frac{e^s-1}{e^s+1}\ \text{and}\ \xi(z)=\frac{z-c}{cz-1}.$$ Then 
    \begin{align*}
        \frac{(e^s+1)z+e^s-1}{(e^s-1)z+e^s+1}=\frac{z+\frac{e^s-1}{e^s+1}}{\frac{e^s-1}{e^s+1}z+1}=\frac{z+c}{cz+1}=\xi(-z).
    \end{align*}Thus, \begin{equation*}C_{\frac{(e^s+1)z+e^s-1}{(e^s-1)z+e^s+1}}=C_{\xi(-z)}=C_{-z}C_\xi\end{equation*} which implies
    \begin{align*}
        C_{\frac{(e^s+1)z+e^s-1}{(e^s-1)z+e^s+1}}T_{\left(\frac{(1-e^s)z+e^s+1}{2e^{s/2}}\right)^{\alpha+2}}&=C_{-z}C_\xi T_{\left(\frac{(1-e^s)z+e^s+1}{2e^{s/2}}\right)^{\alpha+2}}=C_{-z}(T_f C_\xi) 
    \end{align*}
    where 
    \begin{align*}
        f(z)&=\left(\frac{(1-e^s)\xi+e^s+1}{2e^{s/2}}\right)^{\alpha+2}\\
        &=\left(\frac{(1-e^s)\frac{z-\frac{e^s-1}{e^s+1}}{\frac{e^s-1}{e^s+1}z-1}+e^s+1}{2e^{s/2}}\right)^{\alpha+2}\\
        &=\left(\frac{(1-e^s)\frac{(e^s+1)z-(e^s-1)}{(e^s-1)z-(e^s+1)}+e^s+1}{2e^{s/2}}\right)^{\alpha+2}\\
        % &=\left(\frac{(1-e^{2s}+e^{2s}-1)z+(e^s-1)^2-(e^s+1)^2}{2e^{s/2}((e^s-1)z-(e^s+1))}\right)^{\alpha+2}\\
        % &=\left(\frac{-4e^s}{2e^{s/2}((e^s-1)z-(e^s+1)}\right)^{\alpha+2}\\
        &=\left(\frac{2e^{s/2}}{-(e^s-1)z+e^s+1}\right)^{\alpha+2}\\
        &=\left(\frac{\frac{2e^{s/2}}{e^s+1}}{-\frac{e^s-1}{e^s+1}z+1}\right)^{\alpha+2}\\
        &=\left(\frac{\sqrt{1-c^2}}{1-cz}\right)^{\alpha+2}.
    \end{align*}
    Now, $C_{-z}$ is unitary from Theorem 3.2 of \cite{jaoua2010isometric}. Since \begin{equation*}
        \xi(z)=\frac{z-c}{cz-1}=c+\frac{(c^2-1)z}{1-cz}
    \mbox{ and }
        f(z)=\frac{(1-c^2)^{\frac{\alpha+2}{2}}}{(1-cz)^{\alpha+2}}
    \end{equation*}
    for $0<c<1$, we conclude that $T_fC_\xi$ is unitary by Lemma 11 of \cite{cowen2013hermitian}. Therefore $U$ being the composition of two unitary operators is also unitary.    
\end{proof}
Now, we find the Aluthge transform of $C_\phi$ on $\mathcal{A}_\alpha ^2(\mathbb{D)}$ with symbol $\phi(z)=az+(1-a)$ for $0<a<1$.
\begin{theorem}\label{theorem2.2}
     Let $\phi(z)=az+(1-a)$ where $0<a<1$ and $\alpha>-1$. Then $C_\phi = U|C_\phi|$ is the polar decomposition of $C_\phi$ defined on $\mathcal{A}_\alpha ^2(\mathbb{D)}$ where
    $$|C_\phi|=T_{\left(\frac{2\sqrt{a}}{-(1-a)z+(1+a)}\right)^{\alpha+2}}C_{\frac{(-1+3a)z+(1-a)}{-(1-a)z+(1+a)}}$$
    and the unitary operator is given by
    \begin{equation}\label{U=CT}
        U=C_{\frac{(1+a)z+(1-a)}{(1-a)z+(1+a)}}T_{\left(\frac{(a-1)z+(1+a)}{2\sqrt{a}}\right)^{\alpha+2}} \ . 
    \end{equation}
    Also, the unitary operator may be expressed as 
    \begin{equation}\label{U=TC}
        U=T_{\left(\frac{2\sqrt{a}}{(1-a)z+(1+a)}\right)^{\alpha+2}}C_{\frac{(1+a)z+(1-a)}{(1-a)z+(1+a)}}
    \end{equation}
    with the Aluthge transform of $C_\phi$ given by
    $$\widetilde{C_\phi}=T_{\left(\frac{4a}{-(1-a)^2z+(1+a)^2}\right)^{\alpha+2}}C_{\frac{-(1+a)(1-3a)z+(1-a)(1+3a)}{-(1-a)^2z+(1+a)^2}} \ .$$
\end{theorem}
\begin{proof}
    Putting $p=1/2$ in Lemma \ref{Lemma2.2} and by using \eqref{eq3}, we have
    \begin{align*}
        |C_\phi|=e^{s(\alpha+2)/2} T_{\frac{1}{\left(1+\frac{t}{2}-\frac{t}{2}z\right)^{\alpha+2}}}C_{\frac{\frac{t}{2}+(1-\frac{t}{2})z}{1+\frac{t}{2}-\frac{t}{2}z}}=\left(\frac{1}{a}\right)^{(\alpha+2)/2}T_{\frac{2^{\alpha+2}}{(t+2-tz)^{\alpha+2}}}C_{\frac{t+(2-t)z}{2+t-tz}}
    \end{align*}
    where $e^s=1/a$. Substituting $t=e^s-1=\frac{1}{a}-1$, we get
    $$|C_\phi|=T_{\left(\frac{2\sqrt{a}}{-(1-a)z+(1+a)}\right)^{\alpha+2}}C_{\frac{(-1+3a)z+(1-a)}{-(1-a)z+(1+a)}} \ .$$
    Also, from Lemma \ref{lemma2.3},  $$U=C_{\frac{(1+a)z+(1-a)}{(1-a)z+(1+a)}}T_{\left(\frac{(a-1)z+(1+a)}{2\sqrt{a}}\right)^{\alpha+2}}$$ is unitary. It is easy to check that $C_\phi = U|C_\phi|$. The operator $U$ can also be represented as 
    $$U=T_{\left(\frac{2\sqrt{a}}{(1-a)z+(1+a)}\right)^{\alpha+2}}C_{\frac{(1+a)z+(1-a)}{(1-a)z+(1+a)}}.$$
    Now, 
    \begin{eqnarray*}
\widetilde{C_\phi}&=&|C_\phi|^{1/2}U|C_\phi|^{1/2}\\
&=& e^{s(\alpha+2)/2}A_{t/4}UA_{t/4} \ \mbox{(from Lemma \ref{Lemma2.2})}\\
&=& T_{\left(\frac{4a}{-(1-a)^2z+(1+a)^2}\right)^{\alpha+2}}C_{\frac{-(1+a)(1-3a)z+(1-a)(1+3a)}{-(1-a)^2z+(1+a)^2}}.
    \end{eqnarray*}
\end{proof}
Now, we find the iterated Aluthge transform of $C_\phi$ on $\mathcal{A}_\alpha ^2(\mathbb{D)}$ with symbol $\phi(z)=az+(1-a)$ for $0<a<1$.
\begin{theorem}\label{theorem2.3}
     Suppose $\phi(z)=az+(1-a)$ with $0<a<1$ and $\alpha>-1$. For each $n\in \mathbb{N}\cup\{0\}$, the $n$th iterated Aluthge transform of $C_\phi$ on $\mathcal{A}_\alpha ^2(\mathbb{D)}$ is $$\widetilde{C_{\phi}}^{(n)}=T_{f} C_{\psi}$$ 
     where $f$ and $\psi$ are given by
     \begin{equation*} 
         \begin{split}
             f(z)&=\left(\frac{2^{n+1}a^n}{-(1-a)((1+a)^n -2^n a^n)z+ (1+a)((1-a)(1+a)^{n-1}+2^na^n)}\right)^{\alpha+2}, \\
             \psi(z)&=\frac{-(1+a)((1-a)(1+a)^{n-1}-2^na^n)z+(1-a)((1+a)^n+2^na^n)}{-(1-a)((1+a)^n -2^n a^n)z+ (1+a)((1-a)(1+a)^{n-1}+2^na^n)}.
         \end{split}
     \end{equation*}
\end{theorem}
\begin{proof}
    We prove this theorem by induction. For $n=0$, $\widetilde{C_{\phi}}^{(0)}=T_1C_\phi$, hence true. Let us assume that 
    $$\widetilde{C_{\phi}}^{(k)}=T_{f} C_{\psi}$$ 
     where $f$ and $\psi$ are given by
     $$f(z)=\left(\frac{2^{k+1}a^k}{-(1-a)((1+a)^k -2^k a^k)z+ (1+a)((1-a)(1+a)^{k-1}+2^ka^k)}\right)^{\alpha+2},$$
     $$\psi(z)=\frac{-(1+a)((1-a)(1+a)^{k-1}-2^ka^k)z+(1-a)((1+a)^k+2^ka^k)}{-(1-a)((1+a)^k -2^k a^k)z+ (1+a)((1-a)(1+a)^{k-1}+2^ka^k)}$$
     for some non-negative integer $k$.
     It follows that
     \begin{eqnarray}\label{eq5}
         \left(\widetilde{C_{\phi}}^{(k)}\right)^* \widetilde{C_{\phi }}^{(k)} &=& (T_{f } C_{\psi })^*(T_{f } C_{\psi })\nonumber\\
         &=& C_{\psi }^* T_{f }^* T_{f } C_{\psi }\nonumber\\
         &=& T_{\delta } C_{\sigma } T_{f } C_{\psi }\hspace{1cm} \text{(from Lemma \ref{lemma2.1})}\nonumber\\
         &=& T_{\delta } T_{f  \circ \sigma } C_{\sigma } C_{\psi }\nonumber\\
         &=& T_{\delta .(f  \circ \sigma )} C_{\psi  \circ \sigma }
     \end{eqnarray}
     where, 
     $$\delta (z)=\left(\frac{2^{k+1}a^k}{-(1-a)((1+a)^k +2^k a^k)z+ (1+a)((1-a)(1+a)^{k-1}+2^ka^k)}\right)^{\alpha+2}$$
     and
     $$\sigma (z)=\frac{-(1+a)((1-a)(1+a)^{k-1}-2^ka^k)z+(1-a)((1+a)^k-2^ka^k)}{-(1-a)((1+a)^k +2^k a^k)z+ (1+a)((1-a)(1+a)^{k-1}+2^ka^k)}.$$
     Taking $s=\frac{(1-a)(1+a)^k}{2^k a^{k+1}}$, the functions appearing in the symbols have the following terms:
     \begin{align*}
         f (z)=&\left(\frac{2}{-(as -(1-a))z+as +(1+a)}\right)^{\alpha+2},\\
         \psi (z)=&\frac{-(as -(1+a))z+as +(1-a)}{-(as -(1-a))z+as +(1+a)},\\
         \delta (z)=&\left(\frac{2}{-(as +(1-a))z+as +(1+a)}\right)^{\alpha+2},\\
         \sigma (z)=&\frac{-(as -(1+a))z+as -(1-a)}{-(as +(1-a))z+as +(1+a)}.
     \end{align*}
     Now, \begin{equation*}
         \delta (z).(f  \circ \sigma )(z)=\frac{1}{a^{\alpha+2}}\left(\frac{1}{-s  z+1+s }\right)^{\alpha+2}\  \mbox{and }\
         (\psi  \circ \sigma )(z)=\frac{(1-s )z+s }{-s  z+1+s }\ .
     \end{equation*}
     Thus from \eqref{eq5}, we obtain, \begin{align}\label{eq6}
         \left(\widetilde{C_{\phi }}^{(k)}\right)^* \left(\widetilde{C_{\phi }}^{(k)}\right)=\frac{1}{a^{\alpha+2}}T_{(\frac{1}{-s  z+1+s })^{\alpha+2}}C_{\frac{(1-s )z+s }{-s  z+1+s }}=\frac{1}{a^{\alpha+2}}A_{s}
     \end{align}from \eqref{eq3} and hence
     \begin{equation}\label{eq7}
         |\widetilde{C_{\phi }}^{(k)}|=\frac{1}{\sqrt{a^{\alpha+2}}}A_{s }^{1/2}=\frac{1}{\sqrt{a^{\alpha+2}}}A_{\frac{s }{2}}=\frac{1}{\sqrt{a^{\alpha+2}}}T_{(\frac{2}{-s  z+2+s })^{\alpha+2}}C_{\frac{(2-s )z+s }{-s  z+2+s }}
     \end{equation}
     Consequently,
     $$|\widetilde{C_{\phi }}^{(k)}|^{1/2}=\frac{1}{a^{(\alpha+2)/4}}T_{(\frac{4}{-s  z+4+s })^{\alpha+2}}C_{\frac{(4-s )z+s }{-s  z+4+s }}.$$\\
     Now, we claim that $\widetilde{C_{\phi }}^{(k)}=U|\widetilde{C_{\phi }}^{(k)}|$ is the polar decomposition of $\widetilde{C_{\phi }}^{(k)}$ where $U$ is given by \eqref{U=TC}. For that we have,
     \begin{eqnarray*}
          U|\widetilde{C_{\phi }}^{(k)}| &=& \frac{1}{a^{(\alpha+2)/2}}T_{\left(\frac{2\sqrt{a}}{(1-a)z+(1+a)}\right)^{\alpha+2}}C_{\frac{(1+a)z+(1-a)}{(1-a)z+(1+a)}}T_{\left(\frac{2}{-s  z+2+s }\right)^{\alpha+2}}C_{\frac{(2-s )z+s }{-s  z+2+s }}\\
     &=& \frac{1}{a^{(\alpha+2)/2}}T_{\left(\frac{2\sqrt{a}}{(1-a)z+(1+a)}\right)^{\alpha+2}}T_{\left(\frac{2}{-s  \left(\frac{(1+a)z+(1-a)}{(1-a)z+(1+a)}\right)+2+s }\right)^{\alpha+2}}C_{\frac{(2-s )\left(\frac{(1+a)z+(1-a)}{(1-a)z+(1+a)}\right)+s }{-s  \left(\frac{(1+a)z+(1-a)}{(1-a)z+(1+a)}\right)+2+s }}\\
         &=& \frac{1}{a^{(\alpha+2)/2}}T_{\left(\frac{2\sqrt{a}}{(1-a)z+(1+a)}\right)^{\alpha+2}.\left(\frac{2}{-s  \left(\frac{(1+a)z+(1-a)}{(1-a)z+(1+a)}\right)+2+s }\right)^{\alpha+2}}C_{\frac{(2-s )\left(\frac{(1+a)z+(1-a)}{(1-a)z+(1+a)}\right)+s }{-s  \left(\frac{(1+a)z+(1-a)}{(1-a)z+(1+a)}\right)+2+s }}\ .
     \end{eqnarray*}
     Then the symbol of Toeplitz operator is
     \begin{align*}
         & \left(\frac{2\sqrt{a}}{(1-a)z+(1+a)}\right)^{\alpha+2}.\left(\frac{2}{-s  \left(\frac{(1+a)z+(1-a)}{(1-a)z+(1+a)}\right)+2+s }\right)^{\alpha+2}\\
         % =& \left(\frac{4\sqrt{a}}{-s (1+a)z-s (1-a)+(2+s )(1-a)z+(2+s )(1+a)}\right)^{\alpha+2}\\
         =&\left(\frac{4\sqrt{a}}{z(-s -as +2-2a+s -as )+(2+s +2a+as -s +as )}\right)^{\alpha+2}\\
         % =&\left(\frac{4\sqrt{a}}{z(-2as +2-2a)+(2as +2+2a)}\right)^{\alpha+2}\\
         =&\left(\frac{2\sqrt{a}}{-(as -(1-a))z+as  +(1+a)}\right)^{\alpha+2}
         \end{align*}
         and the symbol of composition operator is 
         \begin{align*}
             & \frac{(2-s )\left(\frac{(1+a)z+(1-a)}{(1-a)z+(1+a)}\right)+s }{-s  \left(\frac{(1+a)z+(1-a)}{(1-a)z+(1+a)}\right)+2+s }\\
             % =&\frac{(2-s )(1+a)z+(2-s )(1-a)+s (1-a)z+s (1+a)}{-s (1+a)z-s (1-a)+(2+s )(1-a)z+(2+s )(1+a)}\\
             =&\frac{z(2-s +2a-as +s -as )+(2-s -2a+as +s +as )}{z(-s -as +2+s -2a-as )+(-s +as +2+s +2a+as )}\\
             % =&\frac{z(-2as +2+2a)+(2as +2-2a)}{z(-2as +2-2a)+(2as +2+2a)}\\
             =&\frac{-(as -(1+a))z+as +(1-a)}{-(as -(1-a))z+as +(1+a)}.
         \end{align*}
         Thus\begin{align*}
             U|\widetilde{C_{\phi }}^{(k)}|=&\frac{1}{a^{(\alpha+2)/2}}T_{\left(\frac{2\sqrt{a}}{-(as -(1-a))z+as +(1+a)}\right)^{\alpha+2}}C_\frac{-(as -(1+a))z+as +(1-a)}{-(as -(1-a))z+as +(1+a)}\\=&T_{f } C_{\psi }=\widetilde{C_{\phi }}^{(k)}
         \end{align*}
     where $U=T_{\left(\frac{2\sqrt{a}}{(1-a)z+(1+a)}\right)^{\alpha+2}}C_{\frac{(1+a)z+(1-a)}{(1-a)z+(1+a)}}$ is unitary from Theorem \ref{theorem2.2}. Again,
\begin{eqnarray*}
    \text{ker}\left(\widetilde{C_{\phi }}^{(k)}\right) &=&\text{ker}\left(|\widetilde{C_{\phi }}^{(k)}|\right)\\
    &=&\text{ker}\left(\frac{1}{\sqrt{a^{\alpha+2}}}T_{(\frac{2}{-s  z+2+ s })^{\alpha+2}}C_{\frac{(2-s )z+s }{-s  z+2+s }}\right)\hspace{1cm} \text{(from \eqref{eq7})}\\
    &=&\{0\}.
\end{eqnarray*}
follows from the proof of Theorem 3.4 of \cite{jung2015iterated}.
% (Since, if $T_aC_bh=0$ for any $h \in \mathcal{A}_\alpha^2(\mathbb{D})$, then $a(z)h(b(z))=0$ for any $z \in \mathbb{D}$. If $a\in H^\infty \setminus \{0\}$, then $h(b(z))=0$ for any $z\in \mathbb{D}$. Since $b(\mathbb{D})$ is an open non-empty set by open mapping theorem, thus by identity theorem we have $h\equiv0$ on $\mathbb{D}$. Therfore ker$(T_aC_b)=\{0\}$ for $a\in H^\infty \setminus\{0\}$ and for analytic self map $b$ on $\mathbb{D}$.)\\
    %  \noindent Also from Theorem \ref{theorem2.2},
    % \begin{equation}\label{eq8}
    %     U= C_{\frac{(1+a)z+(1-a)}{(1-a)z+(1+a)}}T_{\left(\frac{(a-1)z+(1+a)}{2\sqrt{a}}\right)^{\alpha+2}}.
    % \end{equation}
    Again from \eqref{U=CT} of Theorem \ref{theorem2.2}, we have 
    \begin{eqnarray*}    \widetilde{C_{\phi}}^{(k+1)}&=&|\widetilde{C_{\phi}}^{(k)}|^{1/2} U |\widetilde{C_{\phi}}^{(k)}|^{1/2}\\
    &=&\left( \frac{1}{a^{(\alpha+2)/4}}T_{\left(\frac{4}{-sz+4+s}\right)^{\alpha+2}}C_{\frac{(4-s)z+s}{-sz+4+s}}\right)\left(C_{\frac{(1+a)z+(1-a)}{(1-a)z+(1+a)}}T_{\left(\frac{(a-1)z+(1+a)}{2\sqrt{a}}\right)^{\alpha+2}}\right)\\
    &&\times \left( \frac{1}{a^{(\alpha+2)/4}}T_{\left(\frac{4}{-sz+4+s}\right)^{\alpha+2}}C_{\frac{(4-s)z+s}{-sz+4+s}}\right)\\
    &=&\frac{1}{\sqrt{a^{\alpha+2}}}\left(T_{\left(\frac{4}{-sz+4+s}\right)^{\alpha+2}}C_{\frac{(4-s)z+s}{-sz+4+s}}\right)\left(C_{\frac{(1+a)z+(1-a)}{(1-a)z+(1+a)}}T_{\left(\frac{(a-1)z+(1+a)}{2\sqrt{a}}\right)^{\alpha+2}}\right)\\
    &&\times \left(T_{\left(\frac{4}{-sz+4+s}\right)^{\alpha+2}}C_{\frac{(4-s)z+s}{-sz+4+s}}\right).
    \end{eqnarray*}
    To simplify the notation, we temporarily set
    \begin{eqnarray*}
        && p(z)= \left(\frac{4}{-sz+4+s}\right)^{\alpha+2},\ b(z)=\frac{(4-s)z+s}{-sz+4+s},\\
        && c(z)=\frac{(1+a)z+(1-a)}{(1-a)z+(1+a)},\ d(z)=\left(\frac{(a-1)z+(1+a)}{2\sqrt{a}}\right)^{\alpha+2}.
    \end{eqnarray*}
    Thus, \begin{eqnarray*}
          \widetilde{C_{\phi}}^{(k+1)}
          &=&\frac{1}{\sqrt{a^{\alpha+2}}} T_{p}C_{b}C_cT_dT_{p}C_{b}\\
          &=& \frac{1}{\sqrt{a^{\alpha+2}}} T_{p} C_{c \circ b} T_{d.p} C_{b}\\
           &=& \frac{1}{\sqrt{a^{\alpha+2}}} T_{p} T_{(d.p)\circ(c \circ b)} C_{c \circ b} C_{b} \\
           &=& \frac{1}{\sqrt{a^{\alpha+2}}} T_{p.((d.p)\circ (c\circ b))} C_{b \circ c \circ b}
      \end{eqnarray*}
      On calculation, we get
      $$(p.((d.p)\circ (c\circ b)))(z)=\left(\frac{4\sqrt{a}}{-(s(1+a)-2(1-a))z+s(1+a)+2(1+a)}\right)^{\alpha+2}$$
      $$(b \circ c \circ b)(z)=\frac{-(s(1+a)-2(1+a))z+s(1+a)+2(1-a)}{-(s(1+a)-2(1-a))z+s(1+a)+2(1+a)}.$$
      Substituting $s=\frac{(1-a)(1+a)^k}{2^k a^{k+1}}$, the Toeplitz symbol and the composition symbol take the form
      $$\left(\frac{2^{k+2}a^{k+1}\sqrt{a}}{-(1-a)((1+a)^{k+1} -2^{k+1} a^{k+1})z+ (1+a)((1-a)(1+a)^{k}+2^{k+1}a^{k+1})}\right)^{\alpha+2}$$and
      $$\frac{-(1+a)((1-a)(1+a)^{k}-2^{k+1}a^{k+1})z+(1-a)((1+a)^{k+1}+2^{k+1}a^{k+1})}{-(1-a)((1+a)^{k+1} -2^{k+1} a^{k+1})z+ (1+a)((1-a)(1+a)^{k}+2^{k+1}a^{k+1})}$$
      respectively. Therefore, we have our proof.
\end{proof}
Using the above theorem and combining Lemma \ref{lemma2.1}, we can express $\widetilde{C_{\phi}}^{(n)}$ as the adjoint of some weighted composition operator as follows:
\begin{corollary}\label{Corollary2.2}
    Suppose $\phi(z)=az+(1-a)$ with $0<a<1$ and $\alpha>-1$. For each $n\in\mathbb{N}\cup\{0\}$, the $n$th iterated Aluthge transform of $C_\phi$ on $\mathcal{A}_\alpha^2 (\mathbb{D})$ is given by $$\widetilde{C_{\phi}}^{(n)}=C_{\Psi}^* T_{F}^*$$ 
     where $$F(z)=\left(\frac{2^{n+1}a^n}{-(1-a)((1+a)^n +2^n a^n)z+ (1+a)((1-a)(1+a)^{n-1}+2^na^n)}\right)^{\alpha+2},$$
     and
     $$\Psi(z)=\frac{-(1+a)((1-a)(1+a)^{n-1}-2^na^n)z+(1-a)((1+a)^n-2^na^n)}{-(1-a)((1+a)^n +2^n a^n)z+ (1+a)((1-a)(1+a)^{n-1}+2^na^n)}.$$
    
\end{corollary}

An operator $T=U|T| \in B(H)$ is said to be quasinormal if $U|T|=|T|U$. An operator $T\in B(H)$ is said to be binormal if $[T^*T,TT^*]=0$ where $[A,B]:=AB-BA$. The set of binormal operators contains the set of quasinormal operators. An operator $T\in B(H)$ is said to be $p$-hyponormal $(0<p<\infty)$ if it satisfies $(T^*T)^p\geq (TT^*)^p$. If $p=1$, $T$ is called hyponormal. Aluthge transform is a good tool for studying various classes of operators, especially $p$-hyponormal operators. 
%For more details, one may refer to \cite{aluthge1990p,bong2000aluthge}.

\begin{corollary}\label{corollary2.2}
    Let $\phi(z)=az+(1-a)$ where $0<a<1$ and $\alpha>-1$. Then $\widetilde{C_\phi}^{(n)}$ defined on $\mathcal{A}_\alpha^2(\mathbb{D)}$ is not quasinormal, but it is binormal for any non-negative integer $n$. Moreover, $C_\phi$ defined on $\mathcal{A}_\alpha^2(\mathbb{D)}$ is not $p$-hyponormal for $p\geq1$.
\end{corollary}
\begin{proof}
    Let $n$ be a non-negative integer. Now by \eqref{K_w},
    $$\langle \widetilde{C_\phi}^{(n)} K_0,K_0\rangle=f(0)=\left(\frac{2^{n+1}a^n}{(1+a)((1-a)(1+a)^{n-1}+2^na^n)}\right)^{\alpha+2}$$ from Theorem \ref{theorem2.3}. Thus, $\langle \widetilde{C_\phi}^{(n)} K_0,K_0\rangle \neq \langle \widetilde{C_\phi}^{(n+1)} K_0,K_0\rangle$, which implies that $\widetilde{C_\phi}^{(n)}$ is not quasinormal by Proposition 1.10 of \cite{bong2000aluthge}. If the polar decomposition of $\widetilde{C_\phi}^{(n)}$ is $U_n|\widetilde{C_\phi}^{(n)}|$, then we have $U_n=U$ for all $n$ from \eqref{U=CT}. Thus from Theorem 3.1 of \cite{ito2004polar}, we have $\widetilde{C_\phi}^{(n)}$ is binormal. Since $\phi(0)\neq 0$, by Theorem 1 of \cite{zorboska1991hyponormal}, $C_\phi$ is not hyponormal on $\mathcal{A}_\alpha^2(\mathbb{D})$. It is well known from \cite{lowner1934monotone} that an operator that is a $p$-hyponormal operator for some $p>1$ is also hyponormal. Therefore, $C_\phi$ is not $p$-hyponormal for $p\geq1$.
\end{proof}

Now, the numerical range $W(T)$ and numerical radius $w(T)$ of an operator $T \in B(H)$ is defined as $ W(T)= \{\langle Tx , x \rangle  : x \in H ,\ || x || =1 \} $ and $w(T)=\{\sup |\lambda|:\lambda \in W(T)\}$. The spectrum of an operator $T$ is denoted by $\sigma(T)$. An operator $T\in B(H)$ is said to be normaloid if $||T||=r(T)$, where $r(T)$ is the spectral radius of $T$. An operator $T \in B(H)$ is said to be a spectraloid if $r(T)=w(T)$.
\begin{corollary}\label{corollary2.3}
     Suppose $\phi(z)=az+(1-a)$ with $0<a<1$, $\alpha>-1$ and $n\in \mathbb{N}\cup\{0\}$. Then, on $\mathcal{A}_\alpha^2(\mathbb{D)}$, the following properties hold:
    \begin{enumerate}[(i)]
    
        \item $0$ lies in the interior of $W(\widetilde{C_\phi}^{(n)})$.
        % \item $\sigma(\widetilde{C_\phi}^{(n)})=\sigma_{ess}(\widetilde{C_\phi}^{(n)})=\left\{\lambda \in \mathbb{C}:|\lambda|\leq \frac{1}{a^{\frac{\alpha+2}{2}}}\right\}$.
        \item $\|\widetilde{C_\phi}^{(n)}\|=\frac{1}{a^{\frac{\alpha+2}{2}}}$ and $\widetilde{C_\phi}^{(n)}$ is normaloid, hence spectraloid.
    \end{enumerate}
\end{corollary}
\begin{proof}
\begin{enumerate}[(i)]
    \item From Corollary \ref{corollary2.2},  $\widetilde{C_{\phi}}^{(n)}$ is non-normal and binormal for all non negative integers $n$. Thus $0\in W(\widetilde{C_{\phi}}^{(n)})$ by Corollary 1 of \cite{embry1970similarities}.
    % \item According to Theorem 8 of \cite{hurst1997relating}, we have 
    % \begin{equation*}
    %     \sigma({C_{\phi}})=\sigma_{ess}(C_\phi)=\{\lambda \in \mathbb{C}:|\lambda|\leq \frac{1}{a^{\frac{\alpha+2}{2}}}\}.
    % \end{equation*}
    % From Theorem 1.3 of \cite{bong2000aluthge}, $\sigma(C_\phi)=\sigma(\widetilde{C_{\phi}}^{(n)})$ and from Theorem 2.3 of \cite{muneo2002spectral}, $\sigma_{ess}(C_\phi)=\sigma_{ess}(\widetilde{C_{\phi}}^{(n)})$. Hence the result follows.
   \item Let $n$ be any non-negative integer. Then 
   \begin{equation*} \left\|\widetilde{C_{\phi}}^{(n)}\right\|^2=\left\|(\widetilde{C_{\phi}}^{(n)})^*\widetilde{C_{\phi}}^{(n)}\right\|=\left\| \frac{1}{a^{\alpha+2}}T_{(\frac{1}{-s  z+1+s })^{\alpha+2}}C_{\frac{(1-s )z+s }{-s  z+1+s }}\right\|\hspace{0.5cm}
   \end{equation*}
   from \eqref{eq6}, where $s=\frac{(1-a)(1+a)^n}{2^n a^{n+1}}$.\\
   Again,
$\left\|T_{(\frac{1}{-s  z+1+s })^{\alpha+2}}C_{\frac{(1-s )z+s }{-s  z+1+s }}\right\|=||A_s||=1$
   by Corollary 18 of \cite{cowen2013hermitian}. Therefore, $\|\widetilde{C_{\phi}}^{(n)}\|=\frac{1}{a^{\frac{\alpha+2}{2}}}$. From Lemma 2.4 of \cite{richman2003subnormality}, $r(C_\phi)=\frac{1}{a^{\frac{\alpha+2}{2}}}$. Since $\sigma(C_\phi)=\sigma(\widetilde{C_{\phi}}^{(n)})$ from Theorem 1.3 of \cite{bong2000aluthge}, $r(\widetilde{C_{\phi}}^{(n)})=\frac{1}{a^{\frac{\alpha+2}{2}}}$. Hence, $\widetilde{C_{\phi}}^{(n)}$ is normaloid. Since, every normaloid operator is spectraloid, $\widetilde{C_{\phi}}^{(n)}$ is spectraloid with $w(\widetilde{C_{\phi}}^{(n)})=r(\widetilde{C_{\phi}}^{(n)})=\frac{1}{a^{\frac{\alpha+2}{2}}}$.
\end{enumerate}
    
\end{proof}
\begin{example}
    For $\alpha>-1$, consider the composition operator $C_\phi$ on $\mathcal{A}_\alpha^2(\mathbb{D)}$ where $\phi(z)=\frac{z+1}{2}$. From Theorem \ref{theorem2.3}, for $n=1$ and $a=1/2$, we conclude that, the Aluthge transform of $C_\phi$ is $\widetilde{C}_\phi=T_\delta C_\gamma$ where $\delta(z)=\left(\frac{8}{-z+9}\right)^{\alpha+2}$ and $\gamma(z)=\frac{3z+5}{-z+9}$. Also from Corollary \ref{corollary2.3}, we have %$\sigma(\widetilde{C}_\phi)=\sigma_{ess}(\widetilde{C}_\phi)=\{\lambda \in \mathbb{C}:|\lambda|\leq 2^{\frac{\alpha+2}{2}}\}$ and
$\|\widetilde{C}_\phi\|=r(\widetilde{C}_\phi)=w(\widetilde{C}_\phi)=2^{\frac{\alpha+2}{2}}$.
\end{example}

\section{Convergence of iterated Aluthge transforms of some composition operator}\label{section3}
In \cite{bong2000aluthge}, Jung et al. conjectured that the sequence of iterated Aluthge transforms of a bounded linear operator converges in the norm topology to a quasinormal operator.
%In \cite{antezana2011iterated}, Antezana et al. have proved the conjecture to be true in the finite dimensional case. In \cite{bong2003iterated}, Jung et al. revised the conjecture and proposed that, for operators on infinite dimensional Hilbert spaces, the sequence of their Aluthge transforms converges in the strong operator topology. Cho et al. in \cite{cho2005aluthge}, have given an example of an unilateral weighted shift for which the sequence of Aluthge iterates does not converge in the weak operator topology. Moreover, they refined the conjecture by stating that for $p-$hyponormal operators with $0<p\leq \infty$, the sequence of iterated Aluthge transforms converges in the strong operator topology.
In \cite{jung2015iterated}, Jung et al. have given partial solutions to this problem on the classical Hardy space. In the following section, we proceed to the same on the weighted Bergman spaces.
\noindent Here, we discuss the convergence of iterated Aluthge transforms of $C_\phi$ on $\mathcal{A}_\alpha^2(\mathbb{D)}$ where $\phi(z)=az+(1-a)$ with $0<a<1$.
\begin{theorem}\label{theorem3.1}
    Let $\phi(z)=az+(1-a)$ with $0<a<1$ and let $\alpha>-1$. Then the sequence $\{\widetilde{C_{\phi}}^{(n)}\}$ defined on $\mathcal{A}_\alpha^2(\mathbb{D)}$ converges to $0$ with respect to the strong operator topology, but fails to converge in the norm topology.
\end{theorem}
\begin{proof}
    From Corollary \ref{corollary2.3}(ii), we conclude that $\{\widetilde{C_{\phi}}^{(n)}\}$ fails to converge to $0$ in the norm topology. To analyze strong convergence, let us take $\omega \in \mathbb{D}$. By Corollary \ref{Corollary2.2}, we get
   $$\widetilde{C_{\phi}}^{(n)}K_\omega=C_{\Psi} ^* T_{F} ^* K_\omega=\overline{F(\omega)}K_{\Psi(\omega)},$$
   where
   \begin{eqnarray*}
       F(z) &=& \left(\frac{2^{n+1}a^n}{-(1-a)((1+a)^n +2^n a^n)z+ (1+a)((1-a)(1+a)^{n-1}+2^na^n)}\right)^{\alpha+2},\\
       \Psi(z) &=& \frac{-(1+a)((1-a)(1+a)^{n-1}-2^na^n)z+(1-a)((1+a)^n-2^na^n)}{-(1-a)((1+a)^n +2^n a^n)z+ (1+a)((1-a)(1+a)^{n-1}+2^na^n)}
   \end{eqnarray*}
   for all non-negative integers $n$ and $K_{\omega}$ is given by \eqref{K_w}. Therefore,   
   \begin{align}\label{eq9}   ||\widetilde{C_{\phi}}^{(n)}K_\omega||^2=|\overline{F(\omega)}|^2\|K_{\Psi(\omega)}\|^2=|F(\omega)|^2\|K_{\Psi(\omega)}\|^2
     \end{align}for all non-negative integers $n$. Now, 
     \begin{eqnarray*}         |F(\omega)|^2&=&|F(\overline{\omega})|^2
     % &=& \left|\frac{2^{n+1}a^n}{-(1-a)((1+a)^n +2^n a^n)\overline{\omega}+ (1+a)((1-a)(1+a)^{n-1}+2^na^n)}\right|^{2(\alpha+2)}\\
         = \left|\frac{2\left(\frac{2a}{1+a}\right)^n}{-(1-a)\left(1 +\left(\frac{2a}{1+a}\right)^n\right)\overline{\omega}+(1-a)+ (1+a)\left(\frac{2a}{1+a}\right)^n }\right|^{2(\alpha+2)}
     \end{eqnarray*}
     and  from \eqref{norm K_w},
     \begin{eqnarray*}
         \|K_{\Psi(\omega)}\|^2&=&\frac{1}{\left(1-\left|{\frac{-(1+a)((1-a)(1+a)^{n-1}-2^na^n)w+(1-a)((1+a)^n-2^na^n)}{-(1-a)((1+a)^n +2^n a^n)w+ (1+a)((1-a)(1+a)^{n-1}+2^na^n)}} \right|^2\right)^{\alpha+2}}\\
         &=&\frac{1}{\left(1-\left|\frac{\left(-(1-a)+(1+a)\left(\frac{2a}{1+a}\right)^n\right)\omega+(1-a)(1-\left(\frac{2a}{1+a}\right)^n}{-(1-a)\left(1+\frac{2a}{1+a}\right)^n\omega+(1-a)+(1+a)\left(\frac{2a}{1+a}\right)^n} \right|^2\right)^{\alpha+2}}
     \end{eqnarray*}for any non-negative integer $n$. Taking $t=\frac{2a}{1+a}$ and since $|\omega|\geq \text{Re}(\omega)$, we have from \eqref{eq9},
     \begin{eqnarray*}
         ||\widetilde{C_{\phi}}^{(n)}K_\omega||^2
         &=& \left(\frac{4t^{2n}}{4t^n(1-a-at^n)|\omega|^2+4t^n(at^n+1-a)-8t^n(1-a)\text{Re}(\omega)}\right)^{\alpha+2}\\
         &\leq&\left(\frac{t^n}{(1-a-at^n)|\omega|^2+(at^n+1-a)-2(1-a)|\omega|}\right)^{\alpha+2}\ \\
         %&=& \left(\frac{t^n}{t^n(-a|\omega|^2+a)+(1-a)|\omega|^2+(1-a)-2(1-a)|\omega|}\right)^{\alpha+2}\\
         %&=&\left(\frac{t^n}{t^n(-a|\omega|^2+a)+(1-a)(|\omega|-1)^2}\right)^{\alpha+2}\\
         &=&\left(\frac{1}{-a|\omega|^2+a+\frac{1}{t^n}(1-a)(|\omega|-1)^2}\right)^{\alpha+2}.
     \end{eqnarray*}
     As $0<t<1$, so $\displaystyle \lim_{n\to \infty}\frac{1}{t^n}=\infty$ and hence $\displaystyle \lim_{n\to \infty}||\widetilde{C_{\phi}}^{(n)}K_\omega||=0$ for all $\omega \in \mathbb{D}$. Since span$\{K_\omega\}$ is dense in $\mathcal{A}_\alpha^2(\mathbb{D)}$ and $\{||\widetilde{C_{\phi}}^{(n)}||\}$ is bounded by Corollary \ref{corollary2.3}(ii), therefore $\{\widetilde{C_{\phi}}^{(n)}\}$ converges to $0$ in the strong operator topology.
\end{proof}

\section{Additional Observation}\label{section4}
In this section, we provide the iterated Aluthge transform of $C_\phi^*$ on $\mathcal{A}_\alpha^2(\mathbb{D)}$ where $\phi(z)=az+(1-a)$ for $0<a<1$. With the help of this, we present the iterated Aluthge transform of $C_\sigma$ where $\sigma(z)=\frac{az}{-(1-a)z+1}$ for $0<a<1$ on some weighted Hardy space. To proceed, we require the following preliminaries on weighted Hardy spaces as found in \cite{maccluer1996composition}.
%Let $\mathcal{H}$ be a Hilbert space of functions analytic on $\mathbb{D}$. If the monomials $1,z,z^2,\ldots$ are an orthogonal set of non-zero vectors with dense span in $\mathcal{H}$, then $\mathcal{H}$ is called a weighted Hardy space. The weight sequence for a weighted Hardy space $\mathcal{H}$ is defined to be $\beta(n)=\|z^n\|$. The weighted Hardy space with weight sequence $\beta(n)$ is denoted by $H^2(\beta)$. The orthogonality implies that the norm on $H^2(\beta)$ is given by $$\left\|\sum_{n=0}^\infty a_nz^n\right\|^2=\sum_{n=0}^\infty |a_n|^2\beta(n)^2$$ and inner product by
%$$\left\langle\sum_{n=0}^\infty a_nz^n\ , \ \sum_{n=0}^\infty b_n z^n\right\rangle=\sum_{n=0}^\infty a_n\overline{b_n}\beta(n)^2.$$
Let $\beta(n)$ be a sequence of positive numbers satisfying 
\begin{eqnarray*}
    \beta(0)=1\ \mbox{and}\ \displaystyle \lim_{n\to \infty} \beta(n)^{1/n}= 1.
\end{eqnarray*}
The weighted Hardy space $H^2(\beta)$ is the space of analytic functions $f$ on $\mathbb{D}$ given by $\displaystyle f(z)=\sum_{n=0}^\infty a_nz^n$, such that 
\begin{align*}
    \|f\|^2=\sum_{n=0}^\infty |a_n|^2\beta(n)^2<\infty.
\end{align*}
Each weighted Hardy space is a Hilbert space with inner product 
$$\left\langle\sum_{n=0}^\infty a_nz^n\ , \ \sum_{n=0}^\infty b_n z^n\right\rangle=\sum_{n=0}^\infty a_n\overline{b_n}\beta(n)^2$$
for which the monomials $\{z^n\}_{n=0}^\infty$ form a complete set of non-zero orthogonal vectors. If $\beta(n)=1$ for all $n$, then $H^2(\beta)$ reduces to the classical Hardy space $H^2$.\\
As defined by Cowen and MacCluer in \cite{cowen1995composition}, the generating function for a weighted Hardy space $H^2(\beta)$ is the function $$k(z)=\sum_{j=0}^\infty\frac{z^j}{\beta(j)^2}.$$
The function $k(z)$ is analytic on the unit disk $\mathbb{D}$ and, for each point $\omega$ in $\mathbb{D}$, the function $K_\omega(z)=k(\overline{\omega}z)$ belongs to $H^2(\beta)$. The $K_\omega$ are the kernel functions for $H^2(\beta)$ satisfying the property $f(\omega)=\langle f,K_\omega \rangle$ and this implies $\|K_\omega\|^2=k(|\omega|^2)$.
\begin{lemma}\label{lemma4.1}
     Suppose $\phi(z)=az+(1-a)$ with $0<a<1$ and $\alpha>-1$. For each $n\in\mathbb{N}\cup\{0\}$, the $n$th iterated Aluthge transform of $C_\phi^*$ on $\mathcal{A}_\alpha^2(\mathbb{D)}$ is $$\widetilde{C_{\phi}^*}^{(n)}=T_{g} C_{\theta}$$ 
     where $g$ and $\theta$ are given by
     \begin{eqnarray*}
         g(z) &=& \left(\frac{2^{n+1}}{-(1-a)((1+a)^n +2^n)z +(1+a)((1-a)(1+a)^{n-1}+2^n)}\right)^{\alpha+2},\\
         \theta(z) &=& \frac{-(1+a)((1-a)(1+a)^{n-1}-2^n)z+(1-a)((1+a)^n-2^n)}{-(1-a)((1+a)^n +2^n)z +(1+a)((1-a)(1+a)^{n-1}+2^n)}.
     \end{eqnarray*}
\end{lemma}
\begin{proof}
    We prove this by using mathematical induction similar to Theorem \ref{theorem2.3} considering the polar decomposition of $\widetilde{C_{\phi }^*}^{(n)}$ as $U'|\widetilde{C_{\phi }^*}^{(n)}|$ where 
    \begin{equation}\label{normAt}
         (\widetilde{C_\phi^*}^{(n)})^*(\widetilde{C_\phi^*}^{(n)})=\frac{1}{a^{\alpha+2}}T_{(\frac{1}{-t  z+1+t })^{\alpha+2}}C_{\frac{(1-t )z+t }{-t  z+1+t }}=\frac{1}{a^{\alpha+2}}A_{t}
     \end{equation} 
     for $t=\frac{(1-a)(1+a)^n}{2^n}$ and the unitary operator $U'$ is given by 
     \begin{equation}\label {U'}
         U'=T_{\left({\frac{2\sqrt{a}}{-(1-a)z+(1+a)}}\right)^{\alpha+2}}C_{\frac{(1+a)z-(1-a)}{-(1-a)z+(1+a)}}
     \end{equation}
\end{proof}
%Applying Theorem D$^\prime$ of \cite{maccluer1996composition} for $\sigma(z)=\frac{az}{-(1-a)z+1}$ where $0<a<1$, we get the following remark.
 %The composition operator $C_\sigma$ is defined similarly on weighted Hardy space. 
%From Lemma \ref{lemma4.1} for $\alpha=0$, we have the iterated Aluthge transform of $C_\phi^*$ on $\mathcal{A}^2(\mathbb{D)}$.
\noindent We now give the iterated Aluthge transform of $C_\sigma$ on some weighted Hardy space where $\sigma(z)=\frac{az}{-(1-a)z+1}$ is the dual map of $\phi(z)=az+(1-a)$ for $0<a<1$. 
\begin{theorem}\label{theorem4.1}
    Let $\phi(z)=az+(1-a)$ and $\sigma(z)=\frac{az}{-(1-a)z+1}$ where $0<a<1$ be maps from  $\mathbb{D}$ into itself. For $\alpha>-1$ and $j\in\mathbb{N}\cup\{0\}$, let 
    \begin{eqnarray}
        \gamma_\alpha(j) &=&\sqrt{\frac{j!\Gamma(\alpha+2)}{\Gamma(j+\alpha+2)}}\label{gamma}\\
        \beta_\alpha(j) &=&\begin{cases}
            1,& j=0\\
            \frac{1}{\gamma_\alpha(j-1)}=\sqrt{\frac{\Gamma(j+\alpha+1)}{{(j-1)!\Gamma(\alpha+2)}}},&j\geq 1.
        \end{cases}\label{beta}
    \end{eqnarray}
    For each non-negative integer $n$, the $n$th iterated Aluthge transform of $C_\sigma$ on $H^2(\beta_\alpha)$ is given by 
    \begin{align}\label{C sigma n}
     \widetilde{C_{\sigma}}^{(n)}=1 \oplus V\left(a\widetilde{C_{\phi}^*}^{(n)}\right)V^*
    \end{align}
    on $H^2(\beta_\alpha)=(zH^2(\beta_\alpha))^\perp \oplus(zH^2(\beta_\alpha))$ and the unitary operator $V:H^2(\gamma_\alpha)\to zH^2(\beta_\alpha)$ is given by 
      \begin{equation}\label{V}
        Vz^j = \frac{\gamma_\alpha(j)}{\beta_\alpha(j+1)}z^{j+1}.
     \end{equation}
\end{theorem}
\begin{proof}
        From the choice of $\gamma_\alpha$ given by \eqref{gamma}, we have $\mathcal{A}_\alpha^2(\mathbb{D})=H^2(\gamma_\alpha)$. Now, $C_\phi$ is a bounded operator on $H^2(\gamma_\alpha)$ and $\{\gamma_\alpha(j)\}$ is a positive sequence satisfying $\displaystyle \lim_{j\to\infty}\frac{\gamma_\alpha(j+1)}{\gamma_\alpha(j)}=1$. Hence by Theorem D$^\prime$ of \cite{maccluer1996composition}, $C_\sigma$ is bounded on $H^2(\beta_\alpha)$, with the restriction of $C_\sigma^*$ to $zH^2(\beta_\alpha)$ unitarily equivalent to $T_a C_\phi$ on $H^2(\gamma_\alpha)$, where $\beta_\alpha$ is given by \eqref{beta}. Using this, we have $$V^*\left(\restr{C_\sigma^*}{zH^2(\beta_\alpha)}\right)V=T_a C_\phi$$
    where the operator $V$ given in \eqref{V} is unitary (see \cite{maccluer1996composition}). Therefore, taking adjoint, we get,
    $$V^*\left(\restr{C_\sigma}{zH^2(\beta_\alpha)}\right)V=T_aC_\phi ^*.$$ Since $\sigma(0)=0$, so $(zH^2(\beta_\alpha))^\perp=$ span$\{1\}$ is a reducing subspace of $C_\sigma$. Thus
    \begin{align*}
        C_\sigma=1 \oplus V\left(aC_\phi ^*\right)V^*
    \end{align*}
    on $(zH^2(\beta_\alpha))^\perp \oplus(zH^2(\beta_\alpha))$. 
    %By Lemma \ref{lemma4.1}, we have the iterated Aluthge transform of $C_\phi^*$ on $\mathcal{A}_\alpha^2(\mathbb{D)}$. Using this, the iterated Aluthge transform of $C_\sigma$ on $H^2(\beta)$ can be computed similarly to the method used in \cite{jung2015iterated} for the classical Hardy space $H^2$. However, since $\beta(n)$ does not correspond to the weight function as in the weighted Bergman space $\mathcal{A}_\alpha^2(\mathbb{D})$, the problem of determining the iterated Aluthge transform of $C_\sigma$ on $\mathcal{A}_\alpha^2(\mathbb{D)}$ still remains open.
    Hence this theorem is true for $n=0$. Assume that 
    $$\widetilde{C_{\sigma}}^{(k)}= 1 \oplus V\left(a\widetilde{C_{\phi}^*}^{(k)}\right)V^*$$for some non-negative integer $k$. Then we get,
    \begin{eqnarray*}
        |\widetilde{C_{\sigma}}^{(k)}|^2 &=& (\widetilde{C_{\sigma}}^{(k)})^*\widetilde{C_{\sigma}}^{(k)}=\left(1 \oplus V\left(a(\widetilde{C_{\phi}^*}^{(k)})^*\right)V^*\right)\left(1 \oplus V\left(a\widetilde{C_{\phi}^*}^{(k)}\right)V^*\right)\\
        &=& 1\oplus V\left(a^2(\widetilde{C_{\phi}^*}^{(k)})^*(\widetilde{C_{\phi}^*}^{(k)})\right)V^*\\
        &=& 1\oplus V\left(a^2|\widetilde{C_{\phi}^*}^{(k)}|^2\right)V^*.\\
    \mbox{Thus, } |\widetilde{C_{\sigma}}^{(k)}| &=& 1\oplus V\left(a|\widetilde{C_{\phi}^*}^{(k)}|\right)V^* \\
\mbox{and } |\widetilde{C_{\sigma}}^{(k)}|^{1/2} &=& 1\oplus V\left(a^{1/2}|\widetilde{C_{\phi}^*}^{(k)}|^{1/2} \right)V^*.
\end{eqnarray*}
    If $\widetilde{C_{\phi}^*}^{(k)}=U'|\widetilde{C_{\phi}^*}^{(k)}|$ is the polar decomposition of $\widetilde{C_{\phi}^*}^{(k)}$ where $U'$ is defined in \eqref{U'}, then 
    \begin{eqnarray*}
        (1\oplus VU'V^*)|\widetilde{C_{\sigma}}^{(k)}|&=&(1\oplus VU'V^*)\left(1\oplus V\left(a|\widetilde{C_{\phi}^*}^{(k)}| \right)V^*\right)\\
        &=& 1\oplus V(a\widetilde{C_{\phi}^*}^{(k)})V^* = \widetilde{C_{\sigma}}^{(k)}.
    \end{eqnarray*}
    Thus $\widetilde{C_{\sigma}}^{(k)}=W|\widetilde{C_{\sigma}}^{(k)}|$ where $W=1\oplus VU'V^*$ is unitary as $U'$ is unitary (from proof of Lemma \ref{lemma4.1}). Also, since $\text{ker}(\widetilde{C_{\phi}^*}^{(k)})=\text{ker}(U')=\{0\}$, so $\text{ker}(\widetilde{C_{\sigma}}^{(k)})=\text{ker}(W)=\{0\}$. Hence, this is the polar decomposition of $\widetilde{C_{\sigma}}^{(k)}$. Now we see that
    \begin{eqnarray*}
        \widetilde{C_{\sigma}}^{(k+1)}&=& |\widetilde{C_{\sigma}}^{(k)}|^{1/2} W |\widetilde{C_{\sigma}}^{(k)}|^{1/2}\\
        &=& \left(1\oplus V\left(a^{1/2}|\widetilde{C_{\phi}^*}^{(k)}|^{1/2} \right)V^*\right)\left(1\oplus VU'V^*\right)\left( 1\oplus V\left(a^{1/2}|\widetilde{C_{\phi}^*}^{(k)}|^{1/2} \right)V^*\right)\\
        &=& 1\oplus V\left(a|\widetilde{C_{\phi}^*}^{(k)}|^{1/2}U' |\widetilde{C_{\phi}^*}^{(k)}|^{1/2} \right)V^*\\
        &=& 1\oplus V\left(a\widetilde{C_{\phi}^*}^{(k+1)}\right)V^*
    \end{eqnarray*}
    Thus, by induction, the result follows for $C_\sigma$ on $H^2(\beta_\alpha)$. %Also, we have from Lemma \ref{lemma4.1},
    %\begin{align*}
     %   \widetilde{C_\sigma}^{(n)}(h(z))=&h(0)\oplus V(aT_gC_\theta)V^*(h(z)-h(0))
    %\end{align*}for any $h\in H^2(\beta)$.
\end{proof}

%An operator $T\in B(H)$ is said to be $p-$hyponormal $(0<p<\infty)$ if it satisfies $(T^*T)^p\geq (TT^*)^p$. If $p=1$, $T$ is called hyponormal. Aluthge transform is a good tool for studying various classes of operators, especially $p-$hyponormal operators. For more details, one may refer to \cite{aluthge1990p,bong2000aluthge}.

In \cite{jung2015iterated}, Jung et al. investigated the iterated Aluthge transform of $C_\sigma$ on the Hardy space $H^2$ where the symbol $\sigma(z)=\frac{az}{-(1-a)z+1}$ for $0<a<1$. In this paper, we consider $\widetilde{C_{\sigma}}^{(n)}$ on some weighted Hardy space $H^2(\beta_\alpha)$ where the weight sequence $\{\beta_\alpha(j)\}$ is given in \eqref{beta}. 
%Since this weight differs from that of the weighted Bergman space $\mathcal{A}_\alpha^2(\mathbb{D})$, the problem of determining the iterated Aluthge transform of $C_\sigma$ on $\mathcal{A}_\alpha^2(\mathbb{D})$ remains open.\\
Now we give some properties of $\widetilde{C_{\sigma}}^{(n)}$ defined on $H^2(\beta_\alpha)$.

\begin{corollary}\label{corollary4.1}
    Let $\sigma(z)=\frac{az}{-(1-a)z+1}$ where $0<a<1$, $\alpha>-1$ and $\beta_\alpha(j)$ be defined by \eqref{beta}. Then $\widetilde{C_{\sigma}}^{(n)}$ defined on $H^2(\beta_\alpha)$ is not quasinormal, hence not normal, but it is binormal for any non-negative integer $n$. 
%Also, $C_{\sigma}^*$ defined on $H^2(\beta_\alpha)$ is not hyponormal for $\alpha>0$.
\end{corollary}

\begin{proof}
     From Theorem \ref{theorem4.1}, we have     \begin{align*}\widetilde{C_{\sigma}}^{(n)}=1 \oplus V\left(a\widetilde{C_{\phi}^*}^{(n)}\right)V^*
    \end{align*}
    on $(zH^2(\beta_\alpha))^\perp \oplus(zH^2(\beta_\alpha))$ where $\phi(z)=az+(1-a)$ and the unitary operator $V$ given by \eqref{V}. Thus it is sufficient to prove that $\widetilde{C_{\phi}^*}^{(n)}$ defined on $\mathcal{A}_\alpha^2(\mathbb{D})$ is not quasinormal. Now by \eqref{K_w},
    $$\langle \widetilde{C_\phi^*}^{(n)} K_0,K_0\rangle=g(0)=\left(\frac{2^{n+1}}{(1+a)((1-a)(1+a)^{n-1}+2^n)}\right)^{\alpha+2}$$ from Lemma \ref{lemma4.1}. Thus, $\langle \widetilde{C_\phi^*}^{(n)} K_0,K_0\rangle \neq \langle \widetilde{C_\phi^*}^{(n+1)} K_0,K_0\rangle$, which implies that $\widetilde{C_\phi^*}^{(n)}$ is not quasinormal by Proposition 1.10 of \cite{bong2000aluthge}. Now, from the proof of Theorem \ref{theorem4.1} the polar decomposition of $\widetilde{C_\sigma}^{(n)}$ is $W|\widetilde{C_{\sigma}}^{(n)}|$ where $W=1\oplus VU'V^*$ and $U'$ is given by \eqref{U'}. We observe that $W$ is independent of $n$. Thus from Theorem 3.1 of \cite{ito2004polar}, we have $\widetilde{C_\sigma}^{(n)}$ is binormal.
    
    %Now for $\alpha>0$, $\displaystyle \sum_{j=0}^\infty\frac{1}{\beta_\alpha(j)^2}<\infty$. Since $\sigma(z)\neq\lambda z$ for $|\lambda|\leq1$, by Theorem 2 of \cite{zorboska1991hyponormal}. we get that $C_{\sigma}^*$ is not hyponormal.
\end{proof}

\begin{corollary}\label{corollary4.2}
    Let $\sigma(z)=\frac{az}{-(1-a)z+1}$ where $0<a<1$, $\alpha>-1$ and $\beta_\alpha(j)$ be defined by \eqref{beta}. Then for $\widetilde{C_{\sigma}}^{(n)}$ defined on $H^2(\beta_\alpha)$, we have the following properties:
    \begin{enumerate}[(i)]
        \item The interior of $W(\widetilde{C_\sigma}^{(n)})$ contains $0$.
        % \item The spectrum of $\sigma(\widetilde{C_{\sigma}}^{(n)})$ is $\left\{\lambda \in \mathbb{C}:|\lambda|\leq \frac{1}{a^{\alpha/2}}\right\} \cup \{1\}$.
        \item $\|\widetilde{C_\sigma}^{n}\|=\sup\ \left\{1,\frac{1}{a^{\alpha/2}}\right\}$.
        % and $\widetilde{C_\sigma}^{(n)}$ is normaloid, hence spectraloid.
     \end{enumerate}
\end{corollary}

\begin{proof}
\begin{enumerate}[(i)]
    \item It follows from Corollary \ref{corollary4.1} by the same argument given in Corollary \ref{corollary2.3}(i).
    % \item Since $\beta_\alpha(j)=(j)^{\frac{\alpha+1}{2}}$ is equivalent to the weight sequence $\tilde{\beta_\alpha}(j)=(j+1)^{\frac{\alpha+1}{2}}$ jor any non-negative integer $j$, by Corollary 10 of \cite{hurst1997relating}, we have 
    % $$\sigma(C_\sigma)=\left\{\lambda \in \mathbb{C}:|\lambda|\leq \frac{1}{a^{\alpha/2}}\right\} \cup \{1\}.$$
    % From Theorem 1.3 of \cite{bong2000aluthge}, $\sigma(C_\sigma)=\sigma(\widetilde{C_{\sigma}}^{(n)})$, hence the result follows.
    \item Let $n$ be any non-negative integer and $\phi(z)=az+(1-a)$. Now from \eqref{normAt}, we have
   \begin{equation*} \left\|\widetilde{C_{\phi}^*}^{(n)}\right\|^2
   % =\left\|(\widetilde{C_{\phi}^*}^{(n)})^*\widetilde{C_{\phi}^*}^{(n)}\right\|
   % =\left\| \frac{1}{a^{\alpha+2}}T_{(\frac{1}{-t  z+1+t })^{\alpha+2}}C_{\frac{(1-t )z+s }{-t  z+1+t }}\right\|
   =\left\|\frac{1}{a^{\alpha+2}} A_t\right\|
   \end{equation*}
   where $t=\frac{(1-a)(1+a)^n}{2^n}$. Again, $||A_t||=1$
   by Corollary 18 of \cite{cowen2013hermitian}. Therefore, $\|\widetilde{C_{\phi}^*}^{(n)}\|=\frac{1}{a^{\frac{\alpha+2}{2}}}$ and thus $\|a\widetilde{C_{\phi}^*}^{(n)}\|=\frac{1}{a^{\alpha/2}}$ for all $n$. Hence from \eqref{C sigma n}, it follows that $\|\widetilde{C_\sigma}^{n}\|=\sup\ \left\{1,\frac{1}{a^{\alpha/2}}\right\}$ for all $n$. 
   % Since $r(\widetilde{C_\sigma}^{n})=\sup \ \left\{1,\frac{1}{a^{\alpha/2}}\right\}$ from (b), it follows that $\widetilde{C_\sigma}^{(n)}$ is normaloid. $\widetilde{C_\sigma}^{(n)}$ is spectraloid follows from the fact that every normaloid operator is spectraloid.
    \end{enumerate}
\end{proof}

\noindent Now, we discuss the convergence of iterated Aluthge transform of $C_\sigma$ on $H^2(\beta_\alpha)$ where $\sigma(z)=\frac{az}{-(1-a)z+1}$ with $0<a<1$ and $\beta_\alpha(j)$ is defined by \eqref{beta}. For this we need the following corollary.

\begin{corollary}\label{corollary4.3}
     Let $\phi(z)=az+(1-a)$ where $0<a<1$. Then, for $\alpha>-1$, the iterated Aluthge tranform of $C_\phi ^*$ on $\mathcal{A}_\alpha^2(\mathbb{D})$ is $$\widetilde{C_{\phi}^*}^{(n)}=C_{\Theta}^* T_{G}^*$$ 
     where $G$ and $\Theta$ are given by 
     \begin{equation}\label{G Theta}\footnotesize
     \begin{split}
          G(z)&=\left(\frac{2^{n+1}}{-(1-a)((1+a)^n -2^n)z +(1+a)((1-a)(1+a)^{n-1}+2^n)}\right)^{\alpha+2},\\
         \Theta(z)&=\frac{-(1+a)((1-a)(1+a)^{n-1}-2^n)z+(1-a)((1+a)^n +2^n)}{-(1-a)((1+a)^n -2^n)z +(1+a)((1-a)(1+a)^{n-1}+2^n)}.
     \end{split}
     \end{equation}
\end{corollary}
\begin{proof}
     Proof follows by using Lemma \ref{lemma4.1} and Lemma \ref{lemma2.1}.
\end{proof}

\begin{theorem}
    Let $\sigma(z)=\frac{az}{-(1-a)z+1}$ where $0<a<1$, $\alpha>-1$ and $\beta_\alpha(j)$ be defined by \eqref{beta}. Then $\widetilde{C_{\sigma}}^{(n)}$ defined on $H^2(\beta_\alpha)$ converges to the normal operator $$(K_0\otimes K_0)\oplus V\left(\frac{1}{a^{\alpha/2}}T_{\left({\frac{2\sqrt{a}}{-(1-a)z+(1+a)}}\right)^{\alpha+2}}C_{\frac{(1+a)z-(1-a)}{-(1-a)z(1+a)}}\right)V^*$$ with respect to the strong operator topology where $V$ is given by \eqref{V}.
\end{theorem}

\begin{proof}
    We first show the convergence of $a\widetilde{C_{\phi}^*}^{(n)}$ on $\mathcal{A}_\alpha^2(\mathbb{D})$.
    From Corollary \ref{corollary4.3}, we have
    $$\widetilde{C_{\phi}^*}^{(n)}K_\omega=C_{\Theta}^* T_{G}^*K_{\omega}=\overline{G(\omega)}K_{\Theta(\omega)}$$
    where $K_\omega$ is given by \eqref{K_w}. Therefore,
%  \begin{eqnarray*}
% && \widetilde{C_{\phi}^*}^{(n)}K_\omega \\
% &=&\left(\frac{2^{n+1}}{-(1-a)((1+a)^n -2^n)\overline{\omega} +(1+a)((1-a)(1+a)^{n-1}+2^n)}\right)^{\alpha+2}
%  \\&& \left(\frac{1}{1-\frac{-(1+a)((1-a)(1+a)^{n-1}-2^n)\overline{\omega}+(1-a)((1+a)^n +2^n)}{-(1-a)((1+a)^n -2^n)\overline{\omega} +(1+a)((1-a)(1+a)^{n-1}+2^n)}z}\right)^{\alpha+2}\\
% %         %=&\left(\frac{2^{n+1}}{-(1-a)((1+a)^n -2^n)\overline{\omega} +(1+a)((1-a)(1+a)^{n-1}+2^n)+(1+a)((1-a)(1+a)^{n-1}-2^n)\overline{\omega}z-(1-a)((1+a)^n +2^n)z}\right)\\
%     &=& \left(\frac{2}{-(1-a)(t^n -1)\overline{\omega}+(1-a)t^n+(1+a)+(1-a)t^n\overline{\omega}z-(1+a)\overline{\omega}z-(1-a)(t^n+1)z}\right)^{\alpha+2}
    % \end{eqnarray*}
    
    \begin{eqnarray*}
    \begin{split}
    &\widetilde{C_{\phi}^*}^{(n)}K_\omega\\=&\left(\frac{2^{n+1}}{-(1-a)((1+a)^n -2^n)\overline{\omega} +(1+a)((1-a)(1+a)^{n-1}+2^n)}\right)^{\alpha+2}\\& \left(\frac{1}{1-\frac{-(1+a)((1-a)(1+a)^{n-1}-2^n)\overline{\omega}+(1-a)((1+a)^n +2^n)}{-(1-a)((1+a)^n -2^n)\overline{\omega} +(1+a)((1-a)(1+a)^{n-1}+2^n)}z}\right)^{\alpha+2}\\
        %=&\left(\frac{2^{n+1}}{-(1-a)((1+a)^n -2^n)\overline{\omega} +(1+a)((1-a)(1+a)^{n-1}+2^n)+(1+a)((1-a)(1+a)^{n-1}-2^n)\overline{\omega}z-(1-a)((1+a)^n +2^n)z}\right)\\
      \small =&\left(\frac{2}{-(1-a)(t^n -1)\overline{\omega}+(1-a)t^n+(1+a)+(1-a)t^n\overline{\omega}z-(1+a)\overline{\omega}z-(1-a)(t^n+1)z}\right)^{\alpha+2}
        \end{split}
    \end{eqnarray*}
    by taking $t=\frac{1+a}{2}$.
    Now, \begin{eqnarray*}
         T_{\left({\frac{2\sqrt{a}}{-(1-a)z+(1+a)}}\right)^{\alpha+2}}C_{\frac{(1+a)z-(1-a)}{-(1-a)z(1+a)}}K_{\omega}
         &=&C_{\frac{(1+a)z+(1-a)}{(1-a)z+(1+a)}}^*T_{\left({\frac{2\sqrt{a}}{(1-a)z+(1+a)}}\right)^{\alpha+2}}^*K_\omega \ \mbox{(by Lemma}\ \ref{lemma2.1})\\
         &=&\left(\frac{2\sqrt{a}}{(1-a)\overline{\omega}+(1+a)}\right)^{\alpha+2}\left(\frac{1}{1-{\frac{(1+a)\overline{\omega}+(1-a)}{(1-a)\overline{\omega}+(1+a)}}z}\right)^{\alpha+2}\\
         &=&\left(\frac{2\sqrt{a}}{(1-a)\overline{\omega}+(1+a)-(1+a)\overline{\omega}z-(1-a)z}\right)^{\alpha+2}.
     \end{eqnarray*}
     Since $0<a<1$, we have, $0<t<1$, and therefore, $\displaystyle \lim_{n\to\infty}t^n=0$. Thus, 
     \begin{equation*}
         \left\|a^{\frac{\alpha+2}{2}}\widetilde{C_{\phi}^*}^{(n)}K_\omega-T_{\left({\frac{2\sqrt{a}}{-(1-a)z+(1+a)}}\right)^{\alpha+2}}C_{\frac{(1+a)z-(1-a)}{-(1-a)z(1+a)}}K_\omega\right\|
     \end{equation*} tends to $0$ as $n\to \infty$ for all $\omega \in \mathbb{D}$. Since span$\{K_\omega\}$ is dense in $\mathcal{A}_\alpha^2(\mathbb{D)}$ and $||a^{\frac{\alpha+2}{2}}\widetilde{C_{\phi}^*}^{(n)}||=1$, hence bounded by proof of Corollary \ref{corollary4.2}(ii), therefore $\{a^{\frac{\alpha+2}{2}}\widetilde{C_{\phi}^*}^{(n)}\}$ converges in the strong operator topology to $T_{\left({\frac{2\sqrt{a}}{-(1-a)z+(1+a)}}\right)^{\alpha+2}}C_{\frac{(1+a)z-(1-a)}{-(1-a)z(1+a)}}$ which is unitary from proof of Lemma \ref{lemma4.1}, and hence normal. Thus, $\{a\widetilde{C_{\phi}^*}^{(n)}\}$ converges to the normal operator $\frac{1}{a^{\alpha/2}}T_{\left({\frac{2\sqrt{a}}{-(1-a)z+(1+a)}}\right)^{\alpha+2}}C_{\frac{(1+a)z-(1-a)}{-(1-a)z(1+a)}}$ in the strong operator topology. Also, we have from Theorem \ref{theorem4.1}, for any $h\in H^2(\beta_\alpha)$,
    \begin{eqnarray*}
     \widetilde{C_\sigma}^{(n)}(h(z))&=&h(0)\oplus V\left(a\widetilde{C_{\phi}^*}^{(n)}\right)V^*(h(z)-h(0))
     \end{eqnarray*}
        which converges to
        \begin{eqnarray*}
        && h(0) \oplus V\left(\frac{1}{a^{\alpha/2}}T_{\left({\frac{2\sqrt{a}}{-(1-a)z+(1+a)}}\right)^{\alpha+2}}C_{\frac{(1+a)z-(1-a)}{-(1-a)z(1+a)}}\right)V^*(h(z)-h(0))\\
     &=&\left((K_0\otimes K_0)\oplus V\left(\frac{1}{a^{\alpha/2}}T_{\left({\frac{2\sqrt{a}}{-(1-a)z+(1+a)}}\right)^{\alpha+2}}C_{\frac{(1+a)z-(1-a)}{-(1-a)z(1+a)}}\right)V^*\right)h.
    \end{eqnarray*}
    where the operator $f\otimes g$ is defined by $(f\otimes g)h=\langle h,g \rangle f$ where $f,g$ are elements of a Hilbert space. Thus, $\widetilde{C_\sigma}^{(n)}$ converges to the normal operator $$(K_0\otimes K_0)\oplus V\left(\frac{1}{a^{\alpha/2}}T_{\left({\frac{2\sqrt{a}}{-(1-a)z+(1+a)}}\right)^{\alpha+2}}C_{\frac{(1+a)z-(1-a)}{-(1-a)z(1+a)}}\right)V^*$$ in the strong operator topology.
\end{proof}

We now discuss the convergence of iterated Aluthge transform of $C_\sigma^*$ on $H^2(\beta_\alpha)$.

\begin{corollary}\label{corollary4.4}
    Let $\sigma(z)=\frac{az}{-(1-a)z+1}$ where $0<a<1$, $\alpha>-1$ and $\beta_\alpha(j)$ be defined by \eqref{beta}. Then $\widetilde{C_{\sigma}^*}^{(n)}$ defined on $H^2(\beta_\alpha)$ converges to the normal operator $K_0\otimes K_0$ in the strong operator topology.
\end{corollary}
\begin{proof}
    From proof of Theorem \ref{theorem4.1}, it follows
    \begin{align*}
    \widetilde{C_{\sigma}^*}^{(n)}=1 \oplus V\left(a\widetilde{C_{\phi}}^{(n)}\right)V^*
    \end{align*}
    on $(zH^2(\beta_\alpha))^\perp \oplus(zH^2(\beta_\alpha))$ where $\phi(z)=az+(1-a)$ and the unitary operator $V$ is given by \eqref{V}. For any $h\in H^2(\beta_\alpha)$,
    \begin{eqnarray*}
     \widetilde{C_\sigma^*}^{(n)}(h(z))&=&h(0)\oplus V\left(a\widetilde{C_{\phi}}^{(n)}\right)V^*(h(z)-h(0)).
     \end{eqnarray*}
     From Theorem \ref{theorem3.1}, $\widetilde{C_\phi}^{(n)}$ converges to $0$ as $n\to\infty$. Therefore $\widetilde{C_\sigma^*}^{(n)}(h(z))$ converges to $h(0)=(K_0\otimes K_0)h$ as $n\to\infty$.
    Hence the result follows.
\end{proof}

%\begin{remark}
    %From Corollaries \ref{corollary4.1} and \ref{corollary4.4}, we observe that $C_{\sigma}^*$ defined on $H^2(\beta_\alpha)$ is not hyponormal for $\alpha>0$ but $\widetilde{C_{\sigma}^*}^{(n)}$ converges to the normal operator $K_0\otimes K_0$ in the strong operator topology. Moreover,  for $\phi(z)=az+(1-a)$, since $\phi(0)=1-a\neq 0$, by Theorem 1 of \cite{zorboska1991hyponormal}, $C_\phi$ is not hyponormal on $\mathcal{A}_\alpha^2(\mathbb{D})$, but $\widetilde{C}_\phi^{(n)}$ converges to $0$ in the strong operator topology. It is well known that if $T\in B(H)$ is a $p-$hyponormal operator for $0<p<\infty$, then $T$ is also $q-$hyponormal for every $0<q\leq p$. These classes of operators provide negative answer to the Conjecture 3.6 of \cite{cho2005aluthge} for $p\geq1$.
%\end{remark}
\section{Further directions}\label{section5}
In \cite{bong2000aluthge}, Jung et al. conjectured that the sequence of iterated Aluthge transforms of a bounded linear operator on a separable, complex Hilbert space converges in the norm topology to a quasinormal operator. Later, Antezana et al. \cite{antezana2011iterated} confirmed this conjecture in the finite dimensional setting. However, Thompson (as communicated in Example 5.5 of \cite{bong2003iterated}) constructed an operator whose sequence of iterated Aluthge transforms converges in the strong operator topology but fails to converge in the norm topology, thereby disproving the conjecture in general. Consequently, Jung et al. \cite{bong2003iterated} revised the conjecture and proposed that, for operators on infinite dimensional Hilbert spaces, the sequence of the Aluthge transforms converges in the strong operator topology. This revised conjecture was further challenged when Cho et al. in \cite{cho2005aluthge}, exhibited a unilateral weighted shift for which the sequence of iterates of the Aluthge transforms does not converge even in the weak operator topology.\\
Subsequently, Jung et al. \cite{jung2015iterated} showed that, for some composition operators on the Hardy space, the iterated Aluthge transform converges in the strong operator topology to a normal operator, but does not converge in the norm topology, thus lending support to the revised conjecture raised in \cite{bong2003iterated}. In our work, we establish similar results for certain composition operators on weighted Bergman space and some weighted Hardy space. Motivated by the findings of \cite{jung2015iterated} together with our own, we pose the following questions.
\begin{question}
     Let $\phi(z)=az+(1-a)$ and $\sigma(z)=\frac{az}{-(1-a)z+1}$ where $0<a<1$. Suppose $C_\phi$ and $C_\sigma$ are bounded composition operators on the weighted Hardy space $H^2(\beta)$. Do $\widetilde{C_{\phi}}^{(n)}$ and $\widetilde{C_{\sigma}}^{(n)}$ converge in the strong operator topology to a quasinormal operator?
\end{question}

\begin{question}
    Let $T$ be a bounded composition operator on some weighted Hardy space. Does the sequence $\{\widetilde{T}^{(n)}\}$ converge to a quasinormal operator in the strong operator topology?
\end{question}

\noindent From Theorem 2 of \cite{cowen1988subnormality}, since $\phi(0)\neq0$, the operator $C_\phi$ fails to be hyponormal on the Hardy space $H^2$ and therefore cannot be $p$-hyponormal for $p\geq1$. Also from Lemma 4.8 of \cite{jung2015characterizations}, $C_\phi$ on $H^2$ is not $p$-hyponormal for $0<p\leq1$. Hence $C_\phi$ on $H^2$ is not $p$-hyponormal for $p>0$. This naturally leads us to the following question.

\begin{question}
    From Corollary \ref{corollary2.2}, we have, $C_\phi$ defined on $\mathcal{A}_\alpha^2(\mathbb{D)}$ is not $p$-hyponormal for $p\geq1$. What happens when $0<p<1$?
\end{question}

\section*{Declarations}

\noindent \textit{Author Contributions:} All authors contributed equally to this work. \\
 \textit{Data Availability :} Not applicable.\\
\textit{Conflict of interest:} The authors declare that they have no competing interests.\\

% \vspace{1cm}

% \noindent\textbf{Author Contributions} All authors contributed equally to this work.\\

% \noindent\textbf{Funding} The authors declare that no specific funding was received for this research.\\

% \noindent\textbf{Data Availability} No datasets were generated or analysed during the current study.

    %\bibliographystyle{plain}
    %\bibliography{references}
% \section*{Declarations}
% \textbf{Conflicts of interest}: The authors declare that they have no competing interests.\\
% \textbf{Availability of data:}    Not applicable.\\
    
\bibliographystyle{plain}

\end{document}